\documentclass[12pt,reqno]{amsart}
\usepackage{hyperref}
\usepackage[a4paper,margin=2.5cm]{geometry}

\newcommand{\Var}{\mathrm{Var}\,}
\newcommand{\E}{\mathbb{E}}
\newcommand{\p}{\mathbb{P}}

\newcommand{\C}{\mathbb{C}}

\newcommand*{\ind}[1]{\mathbf{1}_{\{#1\}}}
\newcommand*{\Ind}[1]{\mathbf{1}_{#1}}

\newcommand{\R}{\mathbb{R}}

\newcommand{\id}{\mathrm{Id}}

\newcommand{\un}{{(n)}}
\newcommand{\INCOMP}{\mathrm{Incomp}}
\newcommand{\COMP}{\mathrm{Comp}}
\newcommand{\SPARSE}{\mathrm{Sparse}}
\newcommand{\SPAN}{\mathrm{span}}
\newcommand{\DIST}{\mathrm{dist}}

\newtheorem{lemma}{Lemma}[section]
\newtheorem{theorem}[lemma]{Theorem}
\newtheorem{proposition}[lemma]{Proposition}
\newtheorem{remark}[lemma]{Remark}

\title[Circular law for log-concave random matrices]{Circular law for random
  matrices with unconditional log-concave distribution}

\thanks{Support: Polish Ministry of Science and Higher Education Iuventus Plus
  Grant no. IP 2011 000171, and French ANR-2011-BS01-007-01 GeMeCoD and
  ANR-08-BLAN-0311-01 Granma.}

\author{Rados{\l}aw Adamczak} %
\address[RA]{Institute of Mathematics, University of Warsaw, Poland} %
\email{R.Adamczak@mimuw.edu.pl}

\author{Djalil Chafa\"{\i}} %
\address[DC]{UMR CNRS 8050 Universit\'e Paris-Est Marne-la-Vall\'ee and Labex
  B\'ezout, France, and Institut Universitaire de France.}
\urladdr{http://djalil.chafai.net/}


\begin{document}

\maketitle

\begin{abstract}
  We explore the validity of the circular law for random matrices with non
  i.i.d.\ entries. Let $A$ be a random $n\times n$ real matrix having as a
  random vector in $\mathbb{R}^{n\times n}$ a log-concave isotropic
  unconditional law. In particular, the entries are uncorellated and have a
  symmetric law of zero mean and unit variance. This allows for some dependence
  and non equidistribution among the entries, while keeping the special case
  of i.i.d.\ standard Gaussian entries. Our main result states that as $n$
  goes to infinity, the empirical spectral distribution of
  $\frac{1}{\sqrt{n}}A$ tends to the uniform law on the unit disc of the
  complex plane.
\end{abstract}

\section{Introduction}

For an $n\times n$ matrix $A$, denote by $\nu_A$ its spectral measure, defined
as $\nu_A = \frac{1}{n} \sum_{i=1}^n \delta_{\lambda_i}$, where
$\lambda_1,\ldots,\lambda_n$ are the eigenvalues of $A$ which are the roots in
$\C$ of the characteristic polynomial of $A$, counted with multiplicities, and
where $\delta_x$ stands for the Dirac mass at point $x$. If $A$ is Hermitian,
then we will treat $\nu_A$ as a finite discrete probability measure on $\R$,
while in the general case it is a finite discrete probability measure on $\C$.
When $A$ is a random matrix, then $\nu_A$ becomes a random measure.

The behaviour of the spectral measure of non-Hermitian random matrices has
drawn considerable attention over the years following the work by Mehta
\cite{Mehta1}, who proved, by using explicit formulas due to Ginibre
\cite{Ginibre}, that the expected spectral measure of $n\times n$ matrices
with i.i.d.\ standard complex Gaussian entries scaled by $\sqrt{n}$ converges
to the uniform measure on the unit disc (which we will call the circular law).
Further developments \cite{G1,G2,Bai,PanZhou,GTCirc,TV} succeeded in extending
the result to a beautiful universal statement valid for any random matrix with
i.i.d.\ entries of unit variance.

It is natural to try to relax the conditions of finite variance and/or
independence imposed on the entries of the matrix by those results. Relaxing
the finite variance assumption leads to limiting distributions which are not
the circular law, see \cite{BCC2,BCAround}. There are various ways to relax
the independence assumption. The circular law was first proved for various
models of random matrices with i.i.d.\ rows, such as for certain random Markov
matrices \cite{BCC1}, for random matrices with i.i.d.\ log-concave rows
\cite{CircLawUnc,InfoNoise}, for random $\pm1$ matrices with i.i.d.\ rows of
given sum \cite{NguyenVuCirc} (see also \cite{MR3010398}), etc. Beyond the
i.i.d.\ rows structure, the circular law was proved only for very specific
models such as for blocs of Haar unitary matrices \cite{MR2480790,MR2919538},
and more recently for random doubly stochastic matrices following the uniform
distribution on the Birkhoff polytope \cite{NguyenDoublyStoch}. Our main
result stated in Theorem \ref{thm:main} below is establishing the circular law
for a new class of random matrices with dependent entries beyond the i.i.d.\
rows structure. Theorem \ref{thm:main} is a natural extension of the model
studied in \cite{CircLawUnc}. However, it does not include as special cases
the models studied in \cite{MR2480790,MR2919538,NguyenDoublyStoch}. The
general idea behind Theorem \ref{thm:main} comes from asymptotic geometric
analysis and states roughly that in large dimensions, for many aspects,
unconditional isotropic log-concave measures behave like product measures with
unit variance and sub-exponential tail.

We say that a probability measure $\mu$ on $\R^d$ is \emph{log-concave} when
for all nonempty compact sets $A,B$ and $\theta \in (0,1)$, $\mu(\theta A +
(1-\theta)B) \ge \mu(A)^\theta\mu(B)^{1-\theta}$. A measure $\mu$ not
supported on a proper affine subspace of $\R^d$ is log-concave iff it has
density $e^{-V}$ where $V \colon \R^d \to \R\cup\{+\infty\}$ is a convex
function, see \cite{B}. We say that $\mu$ is \emph{isotropic} if it is
centered and its covariance matrix is the identity, in other words if the
coordinates are uncorellated, centered, and have unit variance. Recall finally
that $\mu$ is called \emph{unconditional} if $X=(X_1,\ldots,X_d)$ and
$(\varepsilon_1 X_1,\ldots,\varepsilon_d X_d)$ have the same law when $X$ has
law $\mu$ and $\varepsilon_1,\ldots,\varepsilon_d$ are i.i.d.\ Rademacher
random variables (signs) independent of $X$. The isotropy and the
unconditionality are related to the canonical basis of $\R^d$, while the
log-concavity is not. Note that unconditionality together with unit variances
imply automatically isotropy. We use in the sequel the identifications
$\mathcal{M}_n(\R)\equiv\R^{n^2}$ and $\mathcal{M}_n(\C)\equiv\C^{n^2}$ in
which a $n\times n$ matrix $M$ with rows $R_1,\ldots,R_n$ is identified with
the vector $(R_1,\ldots,R_n)$.

\begin{theorem}[Circular law for isotropic unconditional log-concave random
  matrices]\label{thm:main}
  Let $A_n=[X_{ij}^\un]_{1\leq i,j\le n}$ be $n\times n$ random matrices,
  defined on a common probability space. Assume that for each $n$, the
  distribution of $A_n$ as a random vector in $\R^{n^2}$ is log-concave,
  isotropic and unconditional. Then with probability one the empirical
  spectral measure of $\frac{1}{\sqrt{n}}A_n$ converges weakly as $n\to\infty$
  to the uniform measure on the unit disc of $\C$.
\end{theorem}

If one drops the unconditionality assumption in Theorem \ref{thm:main} then
the limiting spectral measure needs not be the circular law. For instance, one
may consider random matrices with density of the form
$A\mapsto\exp(-\mathrm{Tr}(V(\sqrt{AA^*})))$ with $V$ convex and increasing, which are
log-concave (Klein's lemma) but for which the limiting spectral distribution
depends on $V$, see \cite[eq.~(5.8) p.~654]{MR1477381} and \cite{MR2831116}.
This is in contrast with the model of random matrices with i.i.d.\ log-concave
rows studied in \cite{CircLawUnc}, for which it turned out that the circular
law holds without assuming unconditionality, as explained in \cite{InfoNoise}.

We prove Theorem \ref{thm:main} by using the by now classical Hermitization
method introduced by Girko in \cite{G1} and further developed by Tao and Vu in
\cite{TV}. Following the scheme presented in \cite{BCAround}, we first
establish the convergence of the spectral measure of the matrix
\begin{equation}\label{eq:defMnz}
  B_n(z):=\sqrt{\left(\frac{1}{\sqrt{n}}A_n-z\id\right)\left(\frac{1}{\sqrt{n}}A_n-z\id\right)^\ast}
\end{equation}
and later obtain bounds on the small singular values of the matrix, which will
allow us to prove almost sure uniform integrability of the logarithm with
respect to the empirical spectral measure of $B_n(z)$. Putting these
ingredients together we obtain convergence of the logarithmic potential of the
empirical spectral measure of $\frac{1}{\sqrt{n}} A_n$, which ends the proof.

\subsection*{Outline} The organization of the paper is as follows. In Section
\ref{sec:toolbox}, we first recall some basic results concerning log-concave
unconditional measures, which will be useful in the proof. Next in Section
\ref{sec:reduction} we give the outline of the argument, prove convergence of
the empirical spectral measure of $B_n(z)$ and reduce the proof of Theorem
\ref{thm:main} to lower bounds on the singular values of $A_n$. These are
proved in Section \ref{sec:proofs_of_lemmas}.

\subsection*{Notations}
We will denote by $C,c$ positive absolute constants and by $C_a,c_a$ constants
depending only on the parameter $a$. In both cases the values of constants may
differ between occurrences (even in the same line). By $|\cdot|$ we will
denote the standard Euclidean norm of a vector in $\C^n$ or $\R^n$. We will
use $\|\cdot\|$ to denote the operator norm of a matrix and
$\|\cdot\|_{\mathrm{HS}}$ to denote its Hilbert-Schmidt norm. For a
probability measure $\mu$ on $\R$ and $\xi$ in $\C_+ = \{z\in \C\colon \Im z >
0\}$ by $m_\mu(\xi)$ we will denote the Cauchy-Stieltjes transform, i.e.
\[
m_\mu(\xi) = \int_\R \frac{1}{\lambda - \xi}\mu(d\lambda).
\]
We refer to \cite{AGZ,BS} for general theory of Cauchy-Stieltjes transforms
and in particular their connection with weak convergence. For an $n\times n$
matrix $A$ by $s_1(A) \ge \cdots \ge s_n(A)$ we will denote its singular
values, i.e. eigenvalues of $(AA^\ast)^{1/2}$.

\section{Basic facts on unconditional log-concave measures}\label{sec:toolbox}

The geometric and probabilistic analysis of log-concave measures is a well
developed area of research. We recommend the reader the forthcoming monograph
\cite{GiannopoulosBook} for a detailed presentation of this rich theory. What
we will need for our analysis is properties related to the behaviour of
densities and concentration of measure results. In general such questions are
difficult and related to famous open problems, like the
Kannan-Lovasz-Simonovits question \cite{KLS} or the slicing problem
\cite{Hensley, MilmanPajorIsotropic}.

Fortunately, unconditional log-concave measures behave in a much more rigid
way than general ones, in particular some of the aforementioned questions have
been answered either completely or up to terms which are logarithmic in the
dimension and as such do not cause difficulties in the problems we are about
to deal with. Below we present the ingredients we will use throughout the
proof.

The first fact we will need follows immediately from the definition of
log-concavity: linear images of log-concave random vectors are themselves
log-concave. We also have the following tail estimate, which is a special case
of a more general fact due to Borell \cite{B}.

\begin{theorem}[Log-concave random variables have sub-exponential
  tails]\label{thm:Borell} If $X$ is a log-concave, mean zero, variance
  one random variable, then for all $t \ge 0$,
  \[
  \p(|X| \ge t) \le 2\exp(-ct).
  \]
\end{theorem}

The next theorem provides a positive answer to the so-called slicing
conjecture in the case of unconditional log-concave measures. Let us recall
that the conjecture (in one of many equivalent formulations) asserts that the
density of a log-concave isotropic measure in $\R^n$ is bounded by $C^n$.
While this question is wide open in the general case, it has been answered
positively in the case of unconditional measures. We refer for instance to
\cite{BobkovNazarovUnc} for a proof. 
We will use it to obtain sharp small ball inequalities, which will be useful
when dealing with dependence between different rows of the matrix.

\begin{theorem}[Density bound for log-concave measures]\label{thm:Bourgain}
  The density of a log-concave unconditional measure in $\R^n$ is bounded from
  above by $C^n$, where $C$ is a universal constant.
\end{theorem}

Another result we will need is the following version of the Poincar\'e
inequality for log-concave measures, together with the concentration of
measure inequality which follows from it. We refer to
\cite{KlartagUnc,BartheCorderoSymmetry} and to the books
\cite{LedouxConcBook,logSob} for the general theory of concentration of
measure and its relations with functional inequalities. The question whether
all isotropic log-concave measures satisfy the Poincar\'e inequality with a
universal constant is another famous open problem in asymptotic geometric
analysis \cite{KLS}.

\begin{theorem}[Poincar\'e inequality from unconditional
  log-concavity]\label{thm:Poincare}
  If $X$ is an isotropic unconditional log-concave random vector in $\R^n$,
  then for every smooth $f\colon \R^n \to \R$,
  \[
  \Var f(X) \le C\log^2(n+1)  \E|\nabla f(X)|^2.
  \]
\end{theorem}

The above theorem implies in particular concentration of measure inequality
for Lipschitz functions via the so called Herbst argument, see for instance
\cite{MR708367} and \cite{LedouxConcBook}.

\begin{theorem}[Concentration of measure from the Poincar\'e
  inequality]\label{thm:concentration} If a random vector $X$ in $\R^n$
  satisfies the Poincar\'e inequality $\Var f(X) \le \lambda^{-1}\E|\nabla
  f(X)|^2$ for all smooth functions $f \colon \R^n \to \R$, then for all
  1-$Lipschitz$ functions $g\colon \R^n \to \R$ and all $t > 0$,
  \[
  \p(|g(X)-\E g(X)| \ge t) \le 2\exp(-c\sqrt{\lambda}t).
  \]
\end{theorem}

Finally we will need the following result taken from \cite{LatalaWeakDom},
built on previous developments in \cite{BobkovNazarovUnc}. It provides a
comparison of norms of log-concave unconditional random vectors with norms of
vectors with independent exponential coordinates.

\begin{theorem}[Comparison of tails for unconditional log-concave
  measures]\label{thm:Latala}
  If $X$ is an isotropic unconditional log-concave random vector in $\R^n$ and
  $\mathcal{E}$ a standard $n$-dimensional symmetric exponential vector (i.e.
  with i.i.d. components of symmetric exponential distribution with variance
  one), then for any seminorm $\|\cdot\|$ on $\R^n$ and any $t > 0$,
  \[
  \p(\|X\| \ge Ct) \le C\p(\|\mathcal{E}\| \ge t).
  \]
\end{theorem}

\section{Proof of the main result by reduction to singular values
  bounds}\label{sec:reduction}

As already mentioned in the introduction, we will follow the Hermitization
method introduced by Girko, together with a Tao and Vu approach to obtain
lower bounds on the singular values. We refer the reader to \cite{BCAround}
for a presentation of this method in the general case. Let $B_n(z)$ be as in
\eqref{eq:defMnz}. Thanks to \cite[Lemma 4.3]{BCAround}, to prove that with
probability one $\nu_{n^{-1/2}A_n}$ converges weakly to the uniform measure on
the unit disc, it is enough to demonstrate the following two assertions.
\begin{itemize}
\item[(i)] For all $z \in \C$ the spectral measure $\nu_{z,n} := \nu_{B_n(z)}$
  converges almost surely to some deterministic probability measure $\nu_z$ on
  $\R_+$. Moreover, for almost all $z\in \C$,
  \[
    U(z) := -\int_{\R_+} \log (s) \nu_z(ds) =
    \left\{
      \begin{array}{cc}
        -\log |z| & \textrm{if $|z| >1$}\\
        \frac{1}{2}(1-|z|^2) & \textrm{otherwise};
      \end{array}
    \right.
  \]
\item[(ii)] For all $z \in \C$, with probability one the function $s \mapsto
  \log(s)$ is uniformly integrable with respect to the family of measures
  $\{\nu_{z,n}\}_{n\ge 1}$.
\end{itemize}

The quantity $U_n(z) := \int_{\C}\!\log\frac{1}{| \lambda -
  z|}\nu_{\frac{1}{\sqrt{n}}A_n}(d\lambda)$ is the logarithmic potential of
$\nu_{\frac{1}{\sqrt{n}}A_n}$, and we have
$U_n(z)=\int_0^\infty\!\log(s)\,\nu_{B_n(z)}(ds)$, see \cite{BCAround}.
We will first prove point (i).

\begin{proposition}[Singular values of shifts] Assertion (i) above is true.
\end{proposition}

\begin{proof}
  Let us fix $z\in\C$. From Theorem \ref{thm:Poincare} and Theorem
  \ref{thm:concentration} it follows easily that the Euclidean length of a
  random row/column of $A_n$, normalized by $\sqrt{n}$, converges in
  probability to one. Thus we deduce by \cite[Theorem 2.4]{CircLawUnc} that
  the expected spectral measure $\E \nu_{z,n}$ converges weakly, as
  $n\to\infty$, to a probability measure $\nu_z$ which depends only on $z$
  (the identification of $\nu_z$ and the formula involving $U$ can be then
  done and checked on the case of i.i.d.\ Gaussian entries, see for instance
  \cite{BCAround}). The rest of our proof is now devoted to the upgrade to
  almost sure convergence, by using concentration of measure for the
  Cauchy-Stieltjes transform. For now, we know that for every $\xi\in\C_+$,
  \[
  \E m_{\nu_{z,n}}(\xi) %
  = \E \int_{\R_+} \frac{1}{\lambda - \xi} \nu_{z,n}(d\lambda) %
  = m_{\E \nu_{z,n}}(\xi) %
  \underset{n\to\infty}{\longrightarrow} m_{\nu_z}(\xi).
  \]
  At this step, we observe that if $C$ and $C'$ are in $\mathcal{M}_n(\C)$
  with singular values $s_1\geq\cdots\geq s_n$ and $s_1'\geq\cdots\geq s_n'$
  respectively then for every $\xi\in\C_+$,
  \begin{align*}
    |m_{\nu_{\sqrt{CC^*}}}(\xi)-m_{\nu_{\sqrt{C'C'^*}}}(\xi)| &=
    \Big|\frac{1}{n}\sum_{i=1}^{n} \frac{1}{s_i - \xi} -
    \frac{1}{n}\sum_{i=1}^{n}\frac{1}{s_i'-\xi}\Big| \\%
    &\le \frac{1}{n}\sum_{i=1}^{n}%
    \Big|\frac{s_i- s_i'}{(s_i - \xi)(s_i'-\xi)}\Big| \\
    &\le\frac{1}{n|\Im \xi|^2}\sum_{i=1}^{n}|s_i-s_i'|\\
    &\le\frac{1}{\sqrt{n}|\Im \xi|^2}\sqrt{\sum_{i=1}^{n}|s_i-s_i'|^2}\\
    &\le\frac{1}{\sqrt{n}|\Im \xi|^2}\|C-C'\|_{\mathrm{HS}},
  \end{align*}
  where the last step follows from the Hoffman-Wielandt inequality for
  singular values, see for instance \cite[Chapter 4]{lamabook}.
  Thus both the real and the imaginary parts of $m_{\nu_{z,n}}$ are $1/(n(\Im
  \xi)^2)$ Lipschitz functions of $A_n$, with respect to the Hilbert-Schmidt
  norm, which is the Euclidean norm on $\mathcal{M}_n(\R)\equiv\R^{n^2}$.
  Therefore, by Theorem \ref{thm:Poincare} and Theorem
  \ref{thm:concentration}, we get, for every $\xi\in\C_+$ and $\varepsilon >
  0$,
  \[
  \p(|m_{\nu_{z,n}}(\xi) - \E m_{\nu_{z,n}}(\xi)| \ge\varepsilon) %
  \le 2\exp(-cn\varepsilon\Im(z)^2/\log(n)).
  \]
  Now by the first Borel-Cantelli lemma, with probability one,
  $m_{\nu_{z,n}}(\xi) - \E m_{\nu_{z,n}}(\xi)\to0$ as $n\to\infty$ (the set of
  probability one depends on $\xi$). Since $\E m_{\nu_{z,n}}(\xi)\to
  m_{\nu_{z}}(\xi)$ as $n\to\infty$, we get that with probability one,
  $m_{\nu_{z,n}}(\xi)\to m_{\nu_z}(\xi)$ as $n\to\infty$. Since the
  Cauchy-Stieltjes transform is uniformly continuous on every compact subset
  of $\C_+$, it follows that with probability one, for every $\xi\in\C_+$,
  $m_{\nu_{z,n}}(\xi)\to m_{\nu_z}(\xi)$ as $n\to\infty$, which implies
  finally that with probability one, $\nu_{z,n}$ converges weakly to $\nu_z$
  as $n\to\infty$.
\end{proof}

To finish the proof of Theorem \ref{thm:main} it is thus enough to demonstrate
(ii). We will do this using the following three lemmas which give bounds on
singular values of the matrix $\frac{1}{\sqrt{n}}A_n - z\id$. The proofs of
the lemmas will be deferred to the next section. Let us remark that the
formulations we present are in fact more general then what is needed for the
proof of Theorem \ref{thm:main}. We also recall that by $C,c$ we denote
absolute constants.

The first lemma estimates the operator norm of the matrix (largest singular
value).

\begin{lemma}[Largest singular value]\label{le:operator_norm}
  Let $A_n$ be an $n\times n$ random matrix with log-concave isotropic
  unconditional distribution and let $M_n$ be a deterministic $n\times n$
  matrix with $\|M_n\| \le R\sqrt{n}$ for some $R > 0$. Then for all $t \ge
  1$,
  \[
  \p(\|A_n + M_n\| \ge (R+C)\sqrt{n} + t) \le 2\exp(-ct).
  \]
\end{lemma}

Our next lemma provides a bound on the smallest singular value.

\begin{lemma}[Smallest singular value] \label{le:smallest_sv}Let $A_n$ be an
  $n\times n$ random matrix with log-concave isotropic unconditional
  distribution and let $M_n$ be a deterministic $n\times n$ matrix. Then
  \[
  \p(s_n(A_n+M_n) \le  n^{-6.5}) \le Cn^{-3/2}.
  \]
\end{lemma}

\begin{remark} The above lemma is certainly suboptimal, but it is sufficient
  for our applications. In view of the results for Gaussian matrices
  \cite{SSTSmoothed} it is natural to conjecture that for $\varepsilon \in
  (0,1)$, $\p(s_n(A_n+M_n) \le \varepsilon n^{-1/2}) \le C\varepsilon$. For a
  matrix $A_n$ with independent log-concave isotropic columns it is proven in
  \cite{AGLPT2} that $\p(s_n(A_n) \le \varepsilon n^{-1/2}) \le C\varepsilon
  \log^2(2/\varepsilon)$.
\end{remark}

The next lemma we will need gives a bound on the singular values $s_{n-i}$,
where $i > n^{\gamma}$ for some $\gamma \in (0,1)$. It is analogous to an
estimate in \cite{TV} (see also \cite{BCAround}) used to prove the circular
law in the i.i.d. case under minimal assumptions. The main difficulty in its
proof in our setting is lack of independence between the rows of $A_n$, which
has to be replaced by unconditionality and geometric properties implied by
log-concavity.

\begin{lemma}[Count of small singular values]\label{le:small_sv}
  Let $A_n$ be an $n\times n$ random matrix with log-concave isotropic
  unconditional distribution and $M_n$ a deterministic $n\times n$ matrix with
  $\|M_n\| \le R\sqrt{n}$. Let also $\gamma \in (0,1)$. Then for every
  $n^{\gamma} \le i \le n-1$,
  \[
  \p\Big(s_{n-i}(n^{-1/2}(A_n + M_n)) %
  \le c_R \frac{i}{n}\Big) \le 2\exp(-c_{R,\gamma} n^{\gamma/3}).
  \]
\end{lemma}

Let us now finish the proof of Theorem \ref{thm:main}, by demonstrating how
the above lemmas imply point (ii). By Markov's inequality it is enough to show
that for every $z \in \C$, for some small $\alpha > 0$ with probability one
\[
\varlimsup_{n\to\infty} %
\int_0^\infty\!(s^{\alpha} + s^{-\alpha})\nu_{z,n}(ds) < \infty,
\]
or in other words, denoting $s_i := s_i(\frac{1}{\sqrt{n}} A_n - z\id)$,
\[
\varlimsup_{n\to\infty}\frac{1}{n} \sum_{i=1}^n( s_i^\alpha +
s_i^{-\alpha})<\infty.
\]
For all $\alpha \le 2$ we have, for $n$ large enough,
\[
n^{-1}\sum_{i=1}^n s_i^\alpha \le (n^{-1} \sum_{i=1}^n s_i^2)^{\alpha/2} =
(n^{-1/2}\|n^{-1/2} A_n - z\id\|_{\mathrm{HS}})^{\alpha} \le (\|n^{-1/2}A_n -
z\id\|)^\alpha.
\]
But by Lemma \ref{le:operator_norm} and the Borel-Cantelli lemma, with
probability one, we have the bound $\varlimsup_{n\to \infty} \|n^{-1/2}A_n -
z\id\| < C_z$ for some finite constant $C_z$, and thus
\[
\varlimsup_{n\to\infty} \frac{1}{n} \sum_{i=1}^n s_i^\alpha < \infty.
\]

Passing to the other sum, for $\alpha, \gamma$ small enough, by Lemmas
\ref{le:smallest_sv}, \ref{le:small_sv} and the Borel-Cantelli lemma we have
with probability one, for some finite constants $C_{z}$ and
$C_{z,\alpha}$,
\begin{align*}
  \frac{1}{n} \sum_{i=1}^n s_{i}^{-\alpha} &= \frac{1}{n}\sum_{i = 0}^{\lfloor
    n^\gamma\rfloor } s_{n-i}^{-\alpha} %
  + \frac{1}{n} \sum_{i = \lfloor n^\gamma\rfloor +1}^{n-1} s_{n-i}^{-\alpha}\\
  &\le \frac{1}{n}n^{7\alpha} n^{\gamma} %
  + \frac{1}{n}C_{z} \sum_{i=\lfloor n^\gamma\rfloor +1}^{n-1} %
  \Big(\frac{n}{i}\Big)^\alpha \\
  &\le n^{7\alpha + \gamma - 1} + C_{z,\alpha} n^{\alpha - 1}
  n^{1-\alpha} = \mathcal{O}(1).
\end{align*}

This implies (ii) and ends the proof of Theorem \ref{thm:main}.

\section{Proof of singular values bounds}\label{sec:proofs_of_lemmas}

We will start with the proof of Lemma \ref{le:operator_norm}.

\begin{proof}[Proof of Lemma \ref{le:operator_norm}] By the triangle
  inequality it is enough to estimate $\|A_n\|$. By Theorem \ref{thm:Latala}
  we can assume that $A_n$ has i.i.d. entries with the standard symmetric
  exponential distribution. It is well known (see e.g.
  \cite{LatalaMatrixNorm}) that in this case $\E\|A_n\| \le C\sqrt{n}$.
  Moreover, by the Poincar\'e inequality for the product of symmetric
  exponential measures (see e.g. \cite{logSob}) and the fact that the operator
  norm is $1$-Lipschitz with respect to the Hilbert-Schmidt norm, we have
  $\p(\|A_n\| \ge C\E\|A_n\| + t) \le \exp(-ct)$, which allows to finish the
  proof.
\end{proof}

Before we proceed let us introduce some additional notation to be used in the
proofs below. We will assume that $A_n= [X_{ij}^\un]_{i,j\le n}$ but for
notational simplicity we will suppress the superscript $\un$ and write
$X_{ij}$ instead of $X_{ij}^\un$. We will refer to the rows of $A_n$ as
$X_1,\ldots,X_n$. Since the law of $A_n$ is unconditional, sometimes we will
work with the matrix $[\varepsilon_{ij}X_{ij}]_{i,j\le n}$, where
$\varepsilon_{ij}$ are independent Rademacher variables independent of
$X_{ij}$. Slightly abusing the notation, we will sometimes identify this new
matrix with $A_n$.

Let us now pass to the proof of Lemma \ref{le:smallest_sv}.

\begin{proof}[Proof of Lemma \ref{le:smallest_sv}]
  We will denote the rows of $M_n$ and $A_n+M_n$ by $Y_1,\ldots,Y_n$ and
  $Z_1,\ldots,Z_n$ respectively. By estimates from \cite{RV} (see also
  \cite[Lemma B.2]{BCC2}) we have 
  \begin{align}\label{eq:RV_bound}
    \p(s_n(A_n+M_n) \le n^{-6.5}) \le n\max_i\p(\DIST(X_i+Y_i,\SPAN(\{Z_j\}_{j\neq i}))\le n^{-6}).
  \end{align}
  Let us fix $i$. Remarkably, the conditional distribution of $X_{i}$ given
  $(X_j)_{j\neq i}$ is log-concave and unconditional. Let $\sigma_k^2 =
  \E(X_{ik}^2|(X_{j})_{j\neq i})$. Now for every $\varepsilon>0$, every random
  variable $X$ and random vector $Y$, by Markov's inequality,
  $\mathbf{1}_{\{\mathbb{E}(X^2|Y=y)\leq\varepsilon^2\}}\leq\frac{4}{3}\mathbb{P}(|X|\leq
  2\varepsilon\,|\,Y=y)$ for every $y$, which gives
  $\mathbb{P}(\mathbb{E}(X^2|Y)\leq\varepsilon^2)\leq\frac{4}{3}\mathbb{P}(|X|\leq2\varepsilon)$,
  and in particular
  \[
  \p(\sigma_k^2 \le \varepsilon^2) %
  \le \frac{4}{3}\p(|X_{ik}|\le 2\varepsilon).
  \]
  Next, since $X_{ik}$ is log-concave of unit variance, Theorem
  \ref{thm:Bourgain} in dimension one gives
  \[
  \p(\sigma_k^2 \le \varepsilon^2) %
  \le \frac{4}{3}\p(|X_{ik}|\le 2\varepsilon) %
  \le C\varepsilon.
  \]
  Therefore we get
  \[
  \p(\exists_{k\le n} \sigma_k^2 \le n^{-7}) \le Cn\cdot n^{-7/2} = Cn^{-5/2}.
  \]
  Note that $\DIST(X_i+Y_i,\SPAN(\{Z_j\}_{j\neq i})) = |\langle
  X_i+Y_i,e\rangle|$, where $e$ is a random normal to $\SPAN(\{Z_j\}_{j\neq
    i})$ (note that due to the existence of a density this space is with
  probability one of dimension $n-1$). 
  Let $e' = \Re e$, $e'' = \Im e$. Since $|e| = 1$, at least one of the real
  vectors $e',e''$ has Euclidean length greater than $2^{-1/2}$. Without loss of generality we can assume that $|e'| \ge 2^{-1/2}$ (otherwise we may multiply $e$ by $\sqrt{-1}$). By
  unconditionality and the fact that $e'$ is measurable with respect to
  $(X_j)_{j\neq i}$ , we get
  \[
  \E\Big(\langle X_i,e'\rangle^2|(X_j)_{j\neq i}\Big) %
  \ge \frac{1}{2}\min_{k\le n} \sigma_k^2.
  \]
  Moreover, the conditional distribution of $\langle X_i,e'\rangle$ given
  $(X_j)_{j\neq i}$ is log-concave and symmetric. Thus we have
  \begin{align*}
    \p(\DIST(X_i+Y_i,&\SPAN(\{Z_j\}_{j\neq i}))\le n^{-6}) \\
    &= \p(|\langle X_i+Y_i,e\rangle| \le n^{-6}) \\
    &\le \p(|\langle X_i,e'\rangle - \Re \langle Y_i,e\rangle| \le n^{-6})\\
    &\le \p(\exists_{k\le n} \sigma_k^2 \le n^{-7}) + C n^{-6}\E\Big(\E\Big(\langle X_i,e'\rangle^2|(X_j)_{j\neq i}\Big)\Big)^{-1/2}\ind{\forall_k \sigma_k^2> n^{-7}}\\
    &\le Cn^{-5/2} + Cn^{-6}n^{7/2} \le C n^{-5/2},
  \end{align*}
  where we used conditionally the fact that the density of a symmetric
  one-dimensional log-concave r.v. $X$ is bounded by $C/\|X\|_2$ (which
  follows by Theorem \ref{thm:Bourgain}). In combination with
  \eqref{eq:RV_bound} this ends the proof of the lemma.
\end{proof}

It remains to prove Lemma \ref{le:small_sv}. The argument follows the ideas
introduced by Tao and Vu in \cite{TV} and relies on a bound on a distance
between a single row of the matrix and the subspace spanned by some other $k$
rows. Since contrary to the situations considered in \cite{TV,CircLawUnc}, we
do not have independence between rows, a prominent role in the proof will be
played by log-concavity and unconditionality, which will allow us to replace
independence with upper bounds on the densities given in Theorem
\ref{thm:Bourgain}.

\begin{lemma}[Distance to a random subspace]\label{le:distance} Let $A_n$ be
  an $n\times n$ random matrix with log-concave isotropic unconditional
  distribution and $M_n$ a deterministic $n\times n$ matrix with $\|M_n\| \le
  R\sqrt{n}$. Denote the rows of $A_n+M_n$ by $Z_1,\ldots,Z_n$ and let $H$ be
  the space spanned by $Z_1,\ldots,Z_k$ ($k < n$). Then with probability at
  least $1 - 2n\exp(-c_R (n-k)^{1/3})$,
  \[
  \DIST(Z_{k+1},H) \ge c_R \sqrt{n-k}.
  \]
\end{lemma}

Before we prove the above lemma let us show how it implies Lemma
\ref{le:small_sv}. The argument is due to Tao and Vu (see the proof of
\cite[Lemma 6.7]{TV}). We present it here for completeness.

\begin{proof}[Proof of Lemma \ref{le:small_sv}] Consider $i \ge n^\gamma$. Let
  $k = n - \lfloor i/2\rfloor$ and let $B_n'$ be the $k\times n$ matrix with
  rows $Z_1,\ldots,Z_k$ (we use the notation from Lemma \ref{le:distance}). By
  Cauchy interlacing inequalities we have $s_{n-j}(A_n+M_n) \ge s_{n-j}(B_n')$
  for $j\ge \lfloor i/2\rfloor$. Let $H_j$, $j = 1,\ldots,k$ be the subspace
  of $\C^n$ spanned by all the rows of $B_n'$ except for the $j$-th one. By
  \cite[Lemma A4]{TV},
  \[
  \sum_{j=1}^k s_j(B_n')^{-2} = \sum_{j=1}^k \DIST(Z_j,H_j)^{-2}.
  \]
  By Lemma \ref{le:distance}, for each $j\le k$, $\DIST(Z_j,H_j) \ge
  c_R\sqrt{n-k+1}$ with probability at least $1 - 2n\exp(-c_R i^{1/3})$
  (note that we can use the lemma since all its assumptions are preserved
  under permutation of rows of the matrix $A_n$). Thus by the union bound,
  with probability at least $1 - 2n^2\exp(-c_R n^{\gamma/3})$, we get
  \[
  \sum_{j=1}^k s_j(B_n')^{-2} \le C_R \frac{k}{i}.
  \]
  On the other hand, the left-hand side above is at least $s_{n-i}(B_n')^{-2}
  (k - n + i) \ge s_{n-i}(B_n')^{-2} i/2$. This gives that with probability at
  least $1 - 2n^2\exp(-c_R n^{\gamma/3})$,
  \[
  s_{n-i}(A_n+M_n)^2 \ge s_{n-i}(B_n')^2 %
  \ge c_R \frac{ i^2 }{n - \lfloor n^{\gamma}/2\rfloor},
  \]
  which implies that $s_{n-i}(\frac{1}{\sqrt{n}}(A_n+M_n)) \ge
  c_R\frac{i}{n}$. We may now conclude the proof by taking the union bound
  over all $i \ge n^\gamma$ and adjusting the constants.
\end{proof}

\begin{proof}[Proof of Lemma \ref{le:distance}]
  Denote the rows of $M_n$ by $Y_1,\ldots,Y_n$. Recall that thanks to
  unconditionality we can assume that $A_n = [\varepsilon_{ij}X_{ij}]_{i,j\le
    n}$, where $[X_{ij}]_{i,j\le n}$ is log-concave, isotropic and
  unconditional. For simplicity in what follows we will write $\varepsilon_j$
  instead of $\varepsilon_{k+1,j}$.

  With probability one $\dim(H) = k$. Let us however replace $H$ by $K
  =\SPAN(H,Y_{k+1})$ and assume without loss of generality that this space is
  of dimension $k+1$ (if not one can always choose a vector $\tilde{Y}$,
  measurable with respect to $\sigma(Z_1,\ldots,Z_k,Y_{k+1})$ such that $K =
  \SPAN(H,\tilde{Y})$ is of dimension $k+1$). Let $e_1,\ldots,e_{n-k-1}$ be an
  orthonormal basis in $K^\perp$ and $P$ be the orthogonal projection onto
  $K^\perp$. By $e_{ij}$ we will denote the $j$-th coordinate of $e_i$.
  \medskip

  Without loss of generality we can also assume that $k > n/2$ (otherwise we
  may change the constant $c_R$ to $c_R/2$).

  The proof will consist of several steps. First we will take advantage of
  independence between $\varepsilon_{ij}$'s and $X_{ij}$'s and use
  concentration of measure on the discrete cube to provide a lower bound on
  the distance which will be expressed in terms of $X_{ij}$'s and $e_{ij}$'s
  and will hold with high probability with respect to the Rademacher variables
  (i.e. conditionally on $X_{ij}$'s). Next we will use log-concavity to show
  that the random lower bound is itself with high probability bounded from
  below by $c_R\sqrt{n-k}$.

  \medskip

  \noindent\textit{Step 1. Estimates with respect to Rademachers.} We
  have
\begin{align}\label{eq:basis_rep}
  \DIST(X_{k+1},K)
  &= |P X_{k+1}| = \Big(\sum_{i=1}^{n-k-1} |\langle X_{k+1},e_i\rangle|^2\Big)^{1/2}\nonumber\\
  & = \Big(\sum_{i=1}^{n-k-1}\Big|\sum_{j=1}^n
  X_{k+1,j}\varepsilon_{j}\bar{e}_{ij}\Big|^2\Big)^{1/2}
   =: f(\varepsilon_1,\ldots,\varepsilon_n).
\end{align}
The function $f$ is a semi-norm, $L$-Lipschitz with
\[
  L = \sup_{x\in S^{n-1}} |f(x)| = \sup_{x\in S^{n-1}} |P (X_{k+1,i}x_i)_{i=1}^n| \le \sup_{x\in S^{n-1}} \sqrt{\sum_{i=1}^n X_{k+1,i}^2 x_i^2} \le \max_{i\le n}|X_{k+1,i}|.
\]
Moreover, using the Khintchin-Kahane inequality we get
\begin{align*}
  \E_\varepsilon f(\varepsilon_1,\ldots,\varepsilon_n) \ge c (\E_\varepsilon
  f(\varepsilon_1,\ldots,\varepsilon_n)^2)^{1/2} \ge c\Big(\sum_{j=1}^n
  X_{k+1,j}^2\Big(\sum_{i=1}^{n-k-1} |e_{ij}|^2\Big)\Big)^{1/2}.
\end{align*}
Since $\DIST(Z,H) \ge \DIST(X_{k+1},K)$, by Talagrand's concentration inequality on the discrete cube (see \cite{TalCube}) we get for some absolute constant $c> 0$,
\begin{align}\label{eq:conditional_estimate}
&\p_\varepsilon\Big( \DIST(Z_{k+1},H) \le c \Big(\sum_{j=1}^n X_{k+1,j}^2\Big(\sum_{i=1}^{n-k-1} |e_{ij}|^2\Big)\Big)^{1/2}\Big) \nonumber\\
&\le 2\exp\bigg(-c \frac{\sum_{j=1}^n X_{k+1,j}^2\Big(\sum_{i=1}^{n-k-1} |e_{ij}|^2\Big)}{\max_{i\le n} X_{k+1,i}^2}\bigg).\
\end{align}

\medskip
\noindent\textit{Step 2. Lower bounds on coordinates on $e_i$.}
Let $\SPARSE(\delta)$ denote the set of $\delta n$ sparse vectors in $\C^{n}$,
i.e. vectors with at most $\delta n$ nonzero coordinates. Let $S_\C^{n-1}$ be
the unit Euclidean ball in $\C^n$ and define the set of compressible and
incompressible vectors by
\[
\COMP(\delta,\rho) %
= \{x\in S_\C^{n-1}\colon \DIST(x,\SPARSE(\delta)) \le \rho\} %
\quad\text{and}\quad \INCOMP(\delta,\varepsilon) %
= S_\C^{n-1} \setminus \COMP(\delta,\varepsilon).
\]
We will now show that with high probability for each $i\le n- k - 1$, $e_i \in
\INCOMP(\delta,\varepsilon)$ (with $\delta,\varepsilon$ depending only on
$R$).

We will follow the by now standard approach (see \cite{RV,LPRT}) and consider
first the set of sparse vectors. Let $A',M',B'$ be $k\times n$ matrices with
rows resp. $(X_i)_{i\le k},(Y_i)_{i\le k}, (Z_i)_{i\le k}$ and denote by
$X_i',Y_i',Z_i'$ ($i=1,\ldots,n$) their columns. Note that for any real vector
$x \in S^{n-1}$, the random vector
\[
S = \sum_{i=1}^n x_i X_i'
\]
is a log concave isotropic unconditional random vector in $\R^k$ (it is
log-concave as a linear image of a log-concave vector, unconditionality and
isotropicity can be directly verified).
Therefore, by Lemma \ref{thm:Bourgain}, it has a density bounded by $C^k$.
Thus for any deterministic vector $v \in \R^k$,
\begin{align}\label{eq:small_ball}
\p(S \in v + 2r\sqrt{k}B_2^k) \le C^k2^k r^k k^{k/2}vol(B_2^k) \le C^k r^k,
\end{align}
where $B_2^k$ is the unit Euclidean ball in $\R^k$. Consider now any $x \in
S_\C^{n-1}$ and let $x' = \Re x$, $x'' = \Im x$. We have
\[
B'x = \sum_{i=1}^n x_i' X_i' + \Re (\sum_{i=1}^n x_i Y_i') + i\sum_{i=1}^n
x_i'' X_i' + i\Im (\sum_{i=1}^n x_i Y_i').
\]
Setting $v' = -\Re\sum_{i=1}^n x_i Y_i'$ and $v'' = - \Im \sum_{i=1}^n x_i
Y_i'$ we get
\[
  \p(|B' x| \le 2r\sqrt{k}) %
  \le \min\Big(\p( |\sum_{i=1}^n x_i' X_i' - v'|\le 2r\sqrt{k}), \p( |\sum_{i=1}^n x_i'' X_i' - v''|\le 2r\sqrt{k})\Big).
\]
At least one of the vectors $x'$, $x''$ has Euclidean not smaller than
$2^{-1/2}$. Thus using \eqref{eq:small_ball}, we get
\begin{align}\label{eq:individual_vector}
\p(|B' x| \le 2r\sqrt{k}) \le C^kr^k.
\end{align}
Note that the set $\SPARSE(\delta)\cap S_\C^{n-1}$ admits an $\varepsilon$-net
$\mathcal{N}$ of cardinality
\[
\binom{n}{\lfloor \delta n\rfloor }(3/\varepsilon)^{2\lfloor\delta n\rfloor} \le \Big(\frac{C}{\varepsilon^2 \delta}\Big)^{\delta n}.
\]
(This can be easily seen by using volumetric estimates for each choice of the support).
By \eqref{eq:individual_vector} and the union bound with probability at least
\[
1 -r^kC^n\Big(\frac{1}{\varepsilon^2 \delta}\Big)^{\delta n},
\]
for all $x \in \mathcal{N}$,
\[
|B'x| > 2r \sqrt{k}.
\]
If $\varepsilon \le 1/2$ then on the above event we have,
\[
|B'x| > r\sqrt{k} - 2\varepsilon \|B'\|
\]
for all $x \in \COMP(\delta,\varepsilon)$. Indeed if $x \in
\COMP(\delta,\varepsilon)$, then there exists $y \in \SPARSE(\delta)$ such
that $|x-y| \le \varepsilon$. But $|y| \ge 1- \varepsilon$ then and therefore
$|B'y| \ge (1-\varepsilon)|B'\frac{y}{|y|}| \ge (1-\varepsilon)(2r\sqrt{k} -
\varepsilon\|B'\|)$, since $y/|y| \in \SPARSE(\delta)\cap S_\C^{n-1}$ (and so
it can be $\varepsilon$-approximated by a vector from $\mathcal{N}$). Now
$|B'x| \ge |B'y| - \varepsilon \|B'\| \ge r\sqrt{k} - 2\varepsilon\|B'\|$.

Thus by Lemma \ref{le:operator_norm} and the assumption $k > n/2$ we have
\[
  |B'x| > (r/2 - 3(C+R)\varepsilon)\sqrt{n}
\]
for all $x \in \COMP(\delta,\varepsilon)$, with probability at least
\[
  1 - r^{n/2}C^n\Big(\frac{1}{\varepsilon^2 \delta}\Big)^{\delta n} - 2\exp(-c(R+C)\sqrt{n}).
\]
If we set $\varepsilon = r/(12(C+R))$, we get for $r \in (0,1)$ and $\delta$ small enough (depending only on $R$),
\[
|B'x| > r\sqrt{n}/4 > 0
\]
for all $x \in \COMP(\delta,\varepsilon)$,
with probability at least
\begin{align*}
&1 - r^{n/2}C^n\Big(\frac{144(R+C)^2}{r \delta}\Big)^{\delta n} - 2\exp(-c(R+C)\sqrt{n})\\
&\ge 1 - C^n r^{n/4} - 2 \exp(-c(R+C)\sqrt{n}).
\end{align*}
Now for $r$ sufficiently small, the right hand side above is greater than $1 -  2\exp(-c\sqrt{n})$.
In particular we have shown that there exist $\delta,\varepsilon > 0$ depending only on $R$ such that with probability at least $1 - 2\exp(-c\sqrt{n})$,
\[
\inf_{x\in \COMP(\delta,\varepsilon)}|B'x| > 0
\]
and in consequence (since $B' e_i = 0$)
\begin{align}\label{eq:incompressible_basis}
e_i \in \INCOMP(\delta,\varepsilon)
\end{align}
for $i = 1,\ldots,n-k-1$.

\medskip
\noindent\textit{Step 3. Estimating conditional expectation.}
From \cite[Lemma 3.4]{RV} we get that whenever $x \in
\INCOMP(\delta,\varepsilon)$, then there exists a set $I \subseteq \C^n$ of
cardinality at least $\frac{1}{2}\varepsilon^2\delta n$, such that for all $i
\in I$, $|x_i| \ge \frac{\varepsilon}{\sqrt{2n}}$ (the lemma is proved in
\cite{RV} in the real case, but the proof works as well for $\C^n$,
alternatively one may formally pass to the complex case by identifying $\C^n$
with $\R^{2n}$, which will just slightly change the constants).

Using this fact and (\ref{eq:incompressible_basis}), we get that with
probability at least $1 - 2\exp(-c\sqrt{n})$, for $i=1,\ldots,n-k-1$, there
exists a set $I_i \subseteq \{1,\ldots,n\}$ of cardinality at least $\alpha n$
(where $\alpha > 0$ depends only on $R$) such that $|e_{ij}|^2 \ge
\varepsilon^2/2n$ for all $j \in I_i$. We will now prove that for some
constant $\rho > 0$, depending only on $R$, with probability at least $1 -
2\exp(- \alpha n)$, the set $J = \{j\colon |X_{k+1,j}| \ge \rho\}$ satisfies
\begin{align}\label{eq:cardinality_est}
|J| > (1-\alpha/2)n.
\end{align}
Thus we will obtain that for some $\beta$, depending only on $R$, with
probability $1 - 2\exp(- c_R \sqrt{n})$,
\begin{align}\label{eq:cond_exp_est}
  \sum_{j=1}^nX_{k+1,j}^2\sum_{i=1}^{n-k-1}|e_{ij}|^2 %
  \ge \sum_{i=1}^{n-k-1} \sum_{j\in I_i\cap J} X_{k+1,j}^2|e_{ij}^2| %
  \ge (n-k-1)\frac{\alpha n}{2} \rho^2 \frac{\varepsilon^2}{2n} %
  \ge \beta(n-k-1).
\end{align}

To prove \eqref{eq:cardinality_est} we note that for every set $I\subseteq
\{1,\ldots,n\}$ of cardinality $m$, the random vector $(X_{k+1,i})_{i\in I}$
is isotropic, log-concave and unconditional and hence by Theorem
\ref{thm:Bourgain} has a density bounded by $C^m$. Thus
\[
\p(|X_{k+1,i}|\le \rho \;\textrm{for all}\; i\in I) \le 2^mC^m\rho^m.
\]
Taking the union bound over all sets of cardinality $m= \lfloor \alpha
n/2\rfloor $ we obtain
\[
\p(|\{j\colon |X_{k+1,j}| \ge \rho\}| \le (1-\alpha/2)n) \le \binom{n}{m}2^m C^m\rho^m \le2^{\alpha n} C^{\alpha n} e^{C\alpha n\log(2/\alpha)}\rho^{\lfloor \alpha n/2\rfloor},
\]
which gives \eqref{eq:cardinality_est} for $\rho = \exp(-C\log(2/\alpha))$.

\medskip
\noindent\textit{Step 4. Conclusion of the proof.}
Without loss of generality we may assume that $n-k > C$ (otherwise we can make
the bound on probability in the statement of the lemma trivial, by playing
with the constant $c_R$). Note that by Theorem \ref{thm:Borell} and the fact
that $X_{k+1,i}$ is log-concave of mean zero and variance one,
\[
\p(\max_{i\le n} X_{k+1,i}^2 > (n-k)^{2/3}) \le 2n\exp(-c(n-k)^{1/3}).
\]
Combining this estimate with \eqref{eq:cond_exp_est} we get
\[
\p\left(\sum_{j=1}^nX_{k+1,j}^2\sum_{i=1}^{n-k-1}e_{ij}^2 %
  \ge \beta (n-k)/2, \max_{i\le n} X_{k+1,i}^2 %
  \le (n-k)^{2/3}\right) %
\ge 1 -4n\exp(-c_R (n-k)^{1/3}).
\]
Denote the event above by $U$. Using the above estimate together with
\eqref{eq:conditional_estimate} and the Fubini theorem, we get
\begin{align*}
  \p(\DIST(Z_{k+1},H) \le c\sqrt{\beta/2}\sqrt{n-k})
  &\le \E\left(\Ind{U}\p_\varepsilon(\DIST(Z_{k+1},H)\le c\sqrt{\beta/2}\sqrt{n-k})\right) + \p(U^c)\\
  &\le 6n\exp(-c_R(n-k)^{1/3}).
\end{align*}
To prove the lemma it is now enough to adjust the constants.
\end{proof}

\bibliographystyle{abbrv}
\bibliography{cirlog}

\section*{Acknowledgements}

We thank Alice Guionnet for interesting discussions during the winter school
``Random matrices and integrable systems'' held in the Alpine Physics spot
``Les Houches'' (2012).

\end{document}